\documentclass[11pt,a4paper]{amsart}
\usepackage{subfig, graphicx, color,hyperref}
 
\newtheorem{theorem}{Theorem}
\newtheorem{proposition}[theorem]{Proposition}

\theoremstyle{definition}

\newtheorem{remark}[theorem]{Remark}

\newcommand{\CC}{\mathbb{C}}
\newcommand{\NN}{\mathbb{N}}
\newcommand{\RR}{\mathbb{R}}

\newcommand{\textblue}[1]{{\color{blue}{#1}\color{black}}}
\newcommand{\textgreen}[1]{{\color{green}{#1}\color{black}}}
\newcommand{\textred}[1]{{\color{red}{#1}\color{black}}}

\title{A control model for zygodactyl bird's foot}
\author{Anna Chiara Lai \and Paola Loreti}


\begin{document}
\maketitle 

\begin{center}\small
\vskip-0.5cm
\email{anna.lai, paola.loreti@sbai.uniroma1.it}\\
  Dipartimento di Scienze di Base e Applicate per l'Ingegneria\\Sapienza Universit\`a di Roma
\end{center}

\begin{abstract} 
 In this paper we are interested to the  zygodactyly phenomenon in birds, and in particolar in parrots.
This  arrangement, common in  species living on trees,   is a distribution  of  the foot   with two toes facing forward and two back. 
We give a model for the foot,  and thanks to the methods of iterated function system we 
are able to describe the reachability set. Moreover we give a necessary and sufficient condition for the grasping problem.
Finally we introduce a hybrid dynamical system modeling owl's foot in various stages of hunting (flying, attack, grasp). 
\end{abstract}
\section{Introduction}
In this paper we are interested to the  zygodactyly phenomenon in birds, and in particolar in parrots.
This  arrangement, common in  species living on trees,   is a distribution  of  the foot   with two toes facing forward and two back. 
We give a model for the foot,  and thanks to the methods of iterated function system we 
are able to describe the reachability set. Moreover we give a necessary and sufficient condition for the grasping problem.
Finally we introduce a hybrid dynamical system modeling owl's foot in various stages of hunting (flying, attack, grasp).

A discrete dynamical system models the position of the extremal
junction of every finger. A \emph{configuration} is a sequence of states of
the system corresponding to a particular choice for the controls,
while the union of all the possible states of the system is named
\emph{reachable set for the finger}. The closure of the reachable workspace is named
\emph{asymptotic reachable set}. 
Our model includes a control parameter on every phalanx of
every finger of the robot hand, ruling the angle between the current phalanx and the
previous one. Such an angle ranges in $[\pi-\omega,\pi]$, where $\omega$ is a fixed quantity describing the maximal rotation. Note that when the angle is  $\pi$ the phalanx
is consecutive to the previous. We assume a constant ratio $\rho$ between the lengths of two consecutive phalanxes. The structure of the finger ensures the set of possible configurations to be self-similar. In particular the sub-configurations can be looked at as scaled miniatures
with constant ratio $\rho$, named \emph{scaling factor}, of the whole structure. This is the key idea underlying our model and our main tool of investigation.

Biomechanics of avian foot, in particular in the case of arboreal birds, is widely investigated in the literature. 
 We refer to Norberg, (1986) and the references therein for a discussion on the mechanics and energetics of trunk climbing and grasping of the treecreeper.  Bock (1999) describes the morphology
of woodpeckers and the biomechanical analysis of climbing and perching. Zinoviev and Dzerzhinsky(2000) studied the forces acting on avian limbs on various stages of
locomotion, while Sustaita et al. (2013)  survey the tetrapod grasping in several clades, including birds. The above mentioned papers share a mechanical approach to the analysis of grasping and perching
capabilities of avian feet: forces acting on bird's foot are described  by considering in detail the whole skeleton-muscolar system. 
Our model is a simplified version of these systems, on the other hand the dynamical system we consider is indeed a control system: this yields the possibility of investigating at once all the physically
reasonable configurations of the foot. Moreover we shall show that a description of the reachable set of bird's foot can be obtained with an effortable computational cost. 

The discrete control theoretic approach in the investigation of limb's kinematics is common in robotics - among many others, 
we refer to the papers by Imme and Chirikjian (1996),  Lichter,  Sujan and Dubowsky (2002), Lai and Loreti (2012) for an overview on robotic fingers that are
investigated in a fashion similar to the one proposed in the present paper. 
Finally we note that the connection between the biomechanics of avian feet and robotics is an active research domain, mostly motivated by the fact that the lomocotion
of birds turned out to be more efficient with respect with human lomocotion - see for instance the project by Mederreg et al. (2003) where the lomocotion of birds is mimicked in a robotic device.

\section{Preliminaries: Iterated Function Systems}

An iterated function system (IFS) is a set of contractive functions $f_j:\mathbb C\to\mathbb C$. We recall that a function in a metric space $(X,d)$ is a contraction, if for every $x,y\in X$
$$d(f(x),f(y))< c\cdot d(x,y)$$
for some $c<1$.
 Hutchinson (1981) showed that every finite IFS, namely every IFS with finitely many contractions, 
admits a unique non-empty compact fixed point $R$ w.r.t the Hutchinson operator $$\mathcal F: S\mapsto \bigcup_{j=1}^J f_j(S)$$
 Moreover for every non-empty compact set $S\subseteq \mathbb C$
$$\lim_{k\to\infty} \mathcal F^k(S)=R.$$ 
The \emph{attractor} $R$ is a self-similar set and it is the only bounded set satisfying  $\mathcal F(R)=R$. This result was lately generalized to the case of infinite IFS (Mihail and Miculescu. (2009)).
 \section{The finger model on the complex plane}
Recall on physical parameters
\begin{itemize}
\item $\omega\in(0,2\pi)\setminus\{\pi\}$ is the greatest rotation of phalanxes;
\item $\rho>1$ is the scaling factor of phalanxes;
\end{itemize}
Reference parameters
\begin{itemize}
\item $x_0$ is the position of the initial junction of the finger, it is assumed for simplicity coinciding with the origin;
\item $v_0$ represents the initial orientation of the finger: in particular when no rotations are actuated, the finger forms a $-v_0\omega$ angle with the $x$-axis;
\end{itemize}
Control parameters
\begin{itemize}
\item $x_k$ is the position of the $k$-th junction on the complex plane;
\item $v_k\in R=\{r_0=0<a_1<\dots<r_n=1\}$ is the rotation control for the $k$-th phalanx;
\item rotations are modeled with complex exponentiation - see Figure \ref{grafico};
\end{itemize}
\begin{equation}\label{sys}
\begin{cases}
\displaystyle{x_k=x_{k-1}+\frac{1}{\rho^k} e^{-i\sum_{n=0}^{k} v_n~\omega}}\\
\displaystyle{x_0=0}
\end{cases}
\end{equation}
remark that $v_0\in \RR$ is not a control variable but it represent the initial orientation of the finger. 
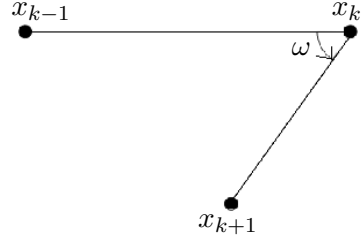
\begin{figure}
\hskip-3cm
\begin{picture}(10,110)
\includegraphics[scale=0.6]{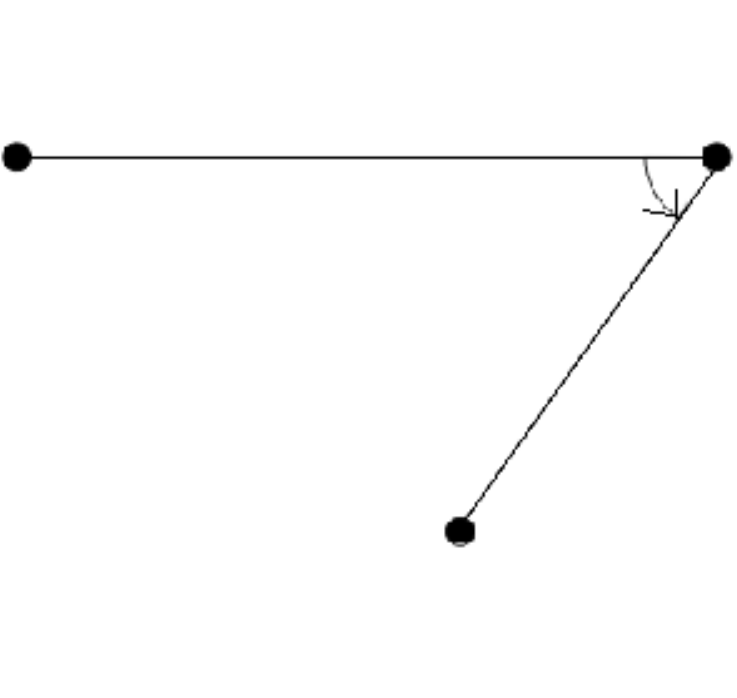}
\put(-130,100){$x_{k-1}$}
\put(-10,100){$x_{k}$}
\put(-60,20){$x_{k+1}$}
\put(-25,85){$\omega$}
\end{picture}
\caption{\label{grafico} In this figure the $k+1$-th rotation control $v_{k+1}$ is equal to $1$: this corresponds to rotate the $k+1$-th phalanx, whose junctions are $x_k$ and $x_{k+1}$ of an angle $\omega$ with respect
to the $k$-phalanx.}
\end{figure}
 The reachable set in time $k$ is 
$$R_k(\rho,\omega,R):=\left\{\sum_{j=1}^k\frac{1}{\rho^j}e^{-i\omega\sum_{n=0}^j v_n}\mid u_j\in E;v_n\in R\right\}.$$
We also define \emph{asymptotically reachable set} the following
\begin{align*}
 &R_\infty(\omega,\rho, R):=\overline{\bigcup_{k=0}^\infty R_k}\\
=&\displaystyle\left\{\sum_{i=1}^\infty\frac{1}{\rho^j}e^{-i\omega\sum_{n=0}^j v_n}\mid u_j\in E;v_n\in R\right\}
\end{align*}
where $\overline R$ denotes the closure of a set $R$. \\

Following proposition shows that the control system (\ref{sys}) and, in particular, its reachable sets are associated to the (possibly infinite) IFS 
$$\mathcal F(\rho,\omega,R):=\left\{f_v: x\mapsto \frac{e^{-i\omega v}}{\rho}(x+1)\mid  v\in R\right\}.$$
\begin{proposition}[Lai and Loreti (2012)]\label{p1}
For every $k\in\NN$ 
\begin{equation}\label{Rk}
R_k(\rho,\omega,R)=e^{-i\omega v_0}\mathcal F^k(\rho,\omega,R)(\{x_0\}). 
\end{equation}
Moreover the asymptotic reachable set satisfies
$$R_\infty(\rho,\omega,R)=e^{-i\omega v_0} R$$
where $R$ is the attractor of $\mathcal F(\rho,\omega,R)$.
\end{proposition}
\begin{proof}[Sketch of proof]
 Equality (\ref{Rk}) can be proved by induction on $k$. Since $e^{i\omega v_0}R_\infty(\rho,\omega)$ is a compact set, it is the attractor if and only if it is invariant with respect to $\mathcal F(\rho,\omega,R)$. This invariance property can be shown by double inclusion.
\end{proof}

\section{Zygodactyly bird's foot}
Zygodactyly bird's foot is composed by four pairwise opposable it occurs parrots, woodpeckers, cuckoos and some owls. 
The arrangement of the fingers varies with species, 
the case we take into account is described in Figure \ref{parrot}.\\
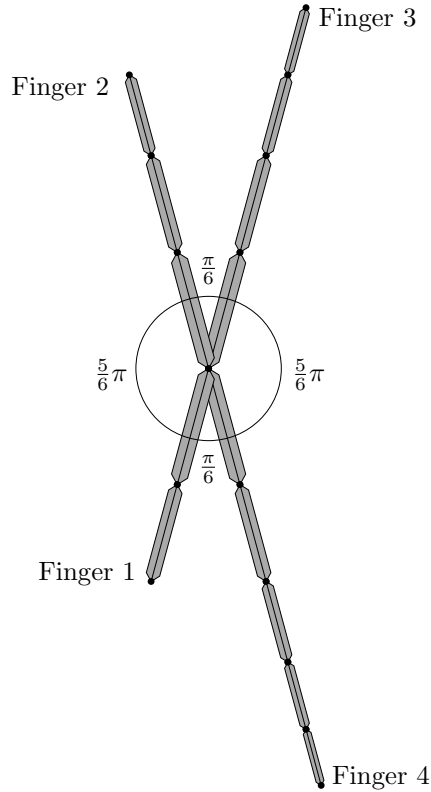
\begin{figure}
\begin{center}
\begin{picture}(110,300)
\put(35,0){\includegraphics[scale=0.28]{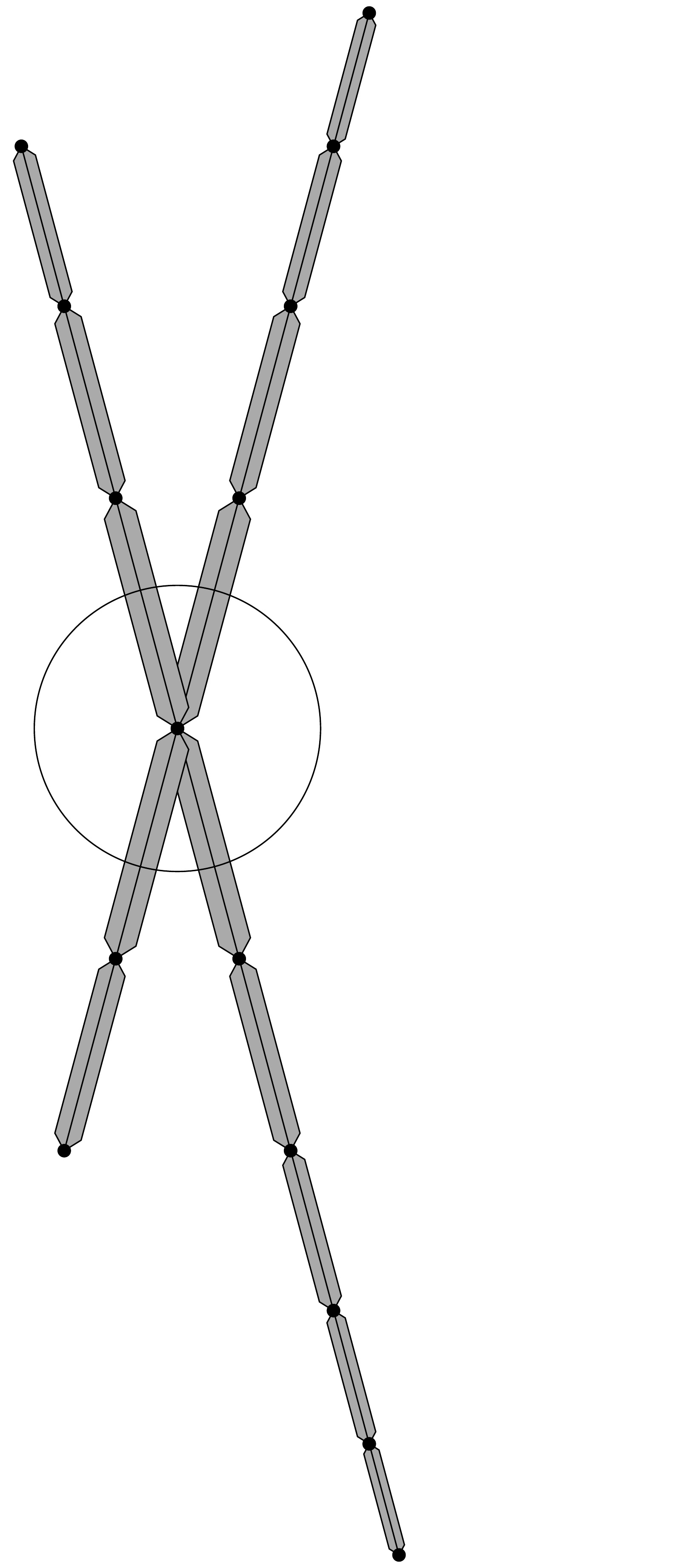}}
\put(65,120){\small $\frac{\pi}{6}$}
\put(65,195){\small $\frac{\pi}{6}$}
\put(100,155){\small $\frac{5}{6}\pi$}
\put(26,155){\small $\frac{5}{6}\pi$}
\put(-5,263){\small Finger 2}
\put(110,289){\small Finger 3}
\put(115,3){\small Finger 4}
\put(5,80){\small Finger 1}
\end{picture}
\end{center}
\caption{\label{parrot} Our model for a zygodactyl bird's foot - the scaling ratio $\rho$ is $1.2$.}
\end{figure}
The junctions of every finger are coplanar, we denote $p_i$, with $i=1,\dots,4$, the plane the $i$-th finger belongs to.
 All the planes of the fingers are assumed to be orthogonal to the $xy$-plane and we call $\Omega_i$, with $i=1,\dots,4$, the angle the plane $p_i$ forms with 
$xz$-plane. In our example $\Omega_1=\Omega_2=\pi/12$, $\Omega_3=\Omega_3=-\pi/12$, so that Finger 1 and Finger 3 (and Finger 2 and Finger 4) are coplanar. The initial rotation $v_0^i$ of the $i$-th finger is set to
$0$ for $i=1,3$, while for $v_0^2=v_0^4=\pi$. All fingers have in common their first junction, that, for seek of simplicity, coincides with the origin. We assume all fingers to have same scaling factor
(in particular $\rho$ is set equal to Golden Mean), the same maximal rotation angle $\omega$ and the same control set $R=[0,1]$, so that 
each phalanx of each finger of the bird's foot can rotate of an angle $v\omega\in[0,\omega]$.

We discuss the reachability of Finger 3, the other cases being similar. By Proposition \ref{p1}, the reachable set of the extremal junction of Finger 3, $R^3_4(\rho,\omega,R)$, can be obtained by following algorithm
\begin{enumerate}
 \item iterate $4$ times the IFS $\mathcal F(\rho,\omega,R)$ with initial datum $\{0\}$;
 \item apply the isomorphism between $\CC$ and $\RR^3\cap xz$-plane given by $x+iz\mapsto(x,0,z)$;
\item apply a rotation of $\Omega_3$ about the $z$-axis. 
\end{enumerate}
See Figure \ref{finger1} for some examples. 

\begin{figure}
\subfloat[$\omega=\pi/12$]{\includegraphics[scale=0.3]{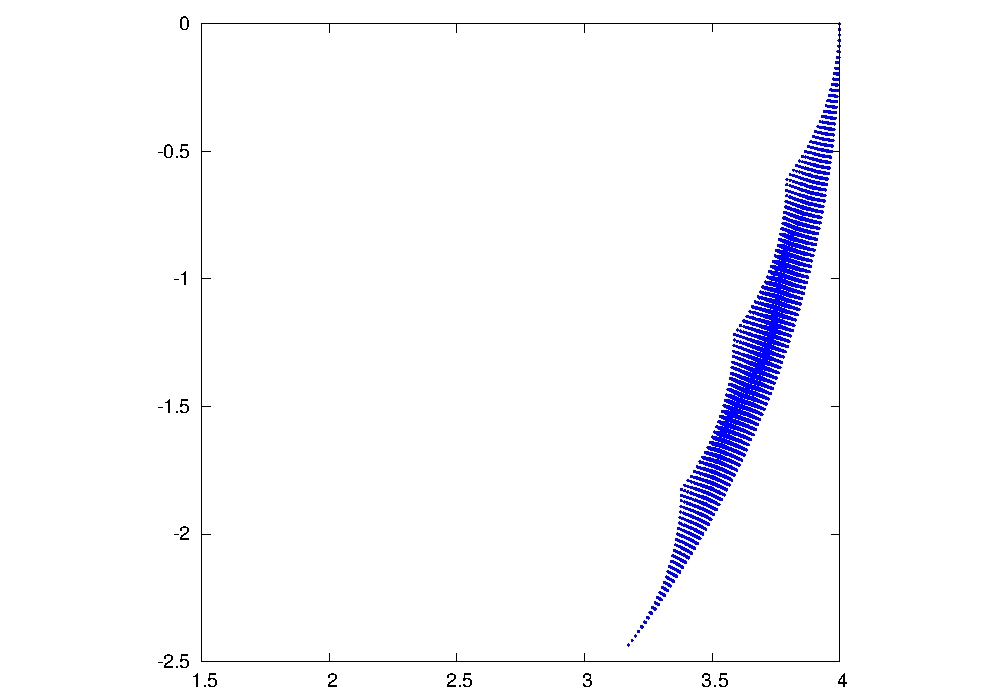}}
\subfloat[$\omega=\pi/6$]{\includegraphics[scale=0.3]{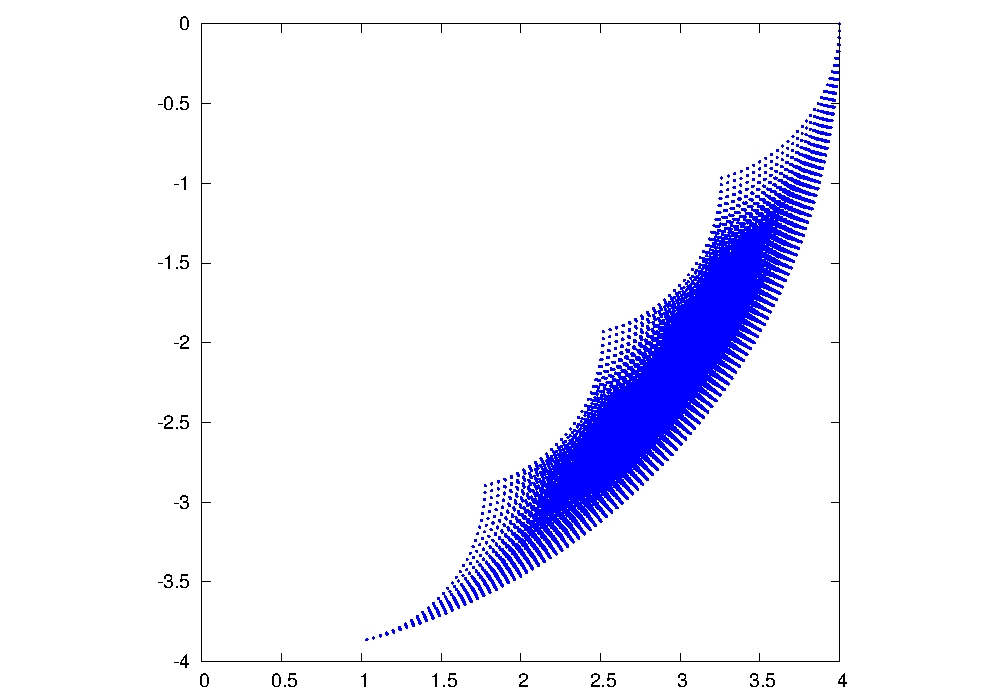}}
 \caption{\label{finger1} An approximation of $R^3_4(\rho,\omega,R)$ for $\omega=\pi/12,\pi/6$ and $\rho$ equal to the Golden Mean, obtained by an uniform discretization of the control set $R=[0,1]$. The consistency of this approximation is given by the continuity in Hausdorff metric
of the attractor of $\mathcal F(\rho,\omega,R)$ with respect to $R$ - see Lai and Loreti (2012).}
\end{figure}

\subsection{Perching on a branch}
We consider the ability of our model bird's foot to lay in a stable equilibrium on branch, modeled as a cylinder. We say that a configuration is \emph{stable on a branch} if at least two phalanxes 
of two opposable fingers are tangent to the cylinder, while by \emph{grasping configuration} we mean a stable configuration where at least a couple of tangent phalanxes have non-positive scalar product. 
Our aim is to describe the possible branches (i.e. cylinders) that can be laid on or grasped by a couple of fingers. We discuss the case of coplanar fingers, say  
Finger 1 and Finger 3, so that the problem can be set on the complex plane and reduces to consider the problem of being stable on/grasping an appropriate ellipse, namely the section of the cylinder-branch related to
the plane $p^1(=p^2)$ - see Figure \ref{plane}.

\begin{figure}
 \begin{picture}(200,170)
 \put(-30,0){\includegraphics[scale=1]{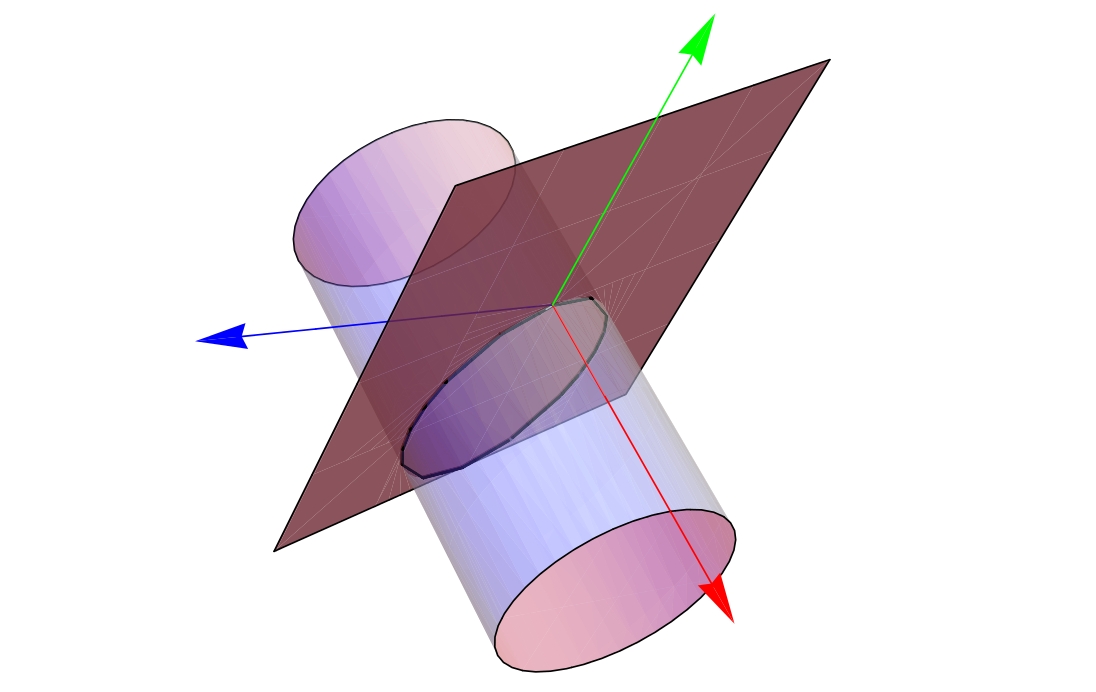}}
\put(150,20){\textred{$y$}}
\put(145,155){\textgreen{$z$}}
\put(165,130){$p^1$}
\put(20,75){\textblue{$x$}}
 \end{picture}
\caption{\label{plane} The cylinder represents a branch and it has radius $1$ and axis $x=0,~z=-1$. The plane $p^1$ is the plane Finger 1 and Finger 3 belong to. The black ellipse is the intersection between $p^1$ and the boundary of the cylinder-branch.} 
\end{figure}
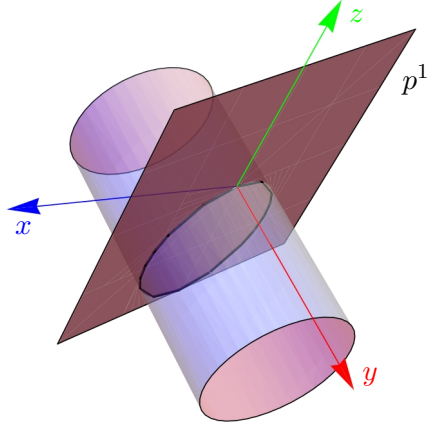

The $k$-th phalanx of the $i$-th finger can be parametrized 
as follows
$$\phi^{(i)}(t;(v_k),k)=x_{k-1}^{(i)}+\frac{t^{(i)}}{\rho^k}e^{-i\sum_{n=0}^k v^{(i)}_nw}$$
with $t^{(i)}\in[0,1]$.

On the other hand the ellipse generated by the intersection of a cylinder (with axis parallel to $y$-axis) and $p^1$ has a generic center of the form $(0,c_y)$ and radii of the form $(r \Omega_1,r)$. 
It can be parametrized as follows 
\begin{align*}
 x(\theta;c_y,r)&= r\cos(\Omega_1) \cos(\theta)\\
 y(\theta; c_y,r)&= c_y+r \sin(\theta)
\end{align*}
and its tangent vector given by is
\begin{align*}
 \dot x(\theta;c_y,r)&= -r\cos(\Omega_1) \sin(\theta)\\
 \dot y(\theta;c_y,r)&= r \cos(\theta)
\end{align*}
Using the classical isomorphism $(x,y)\mapsto x+iy$ the ellipse on the complex plane may be parametrized by 
$$\gamma(\theta;c_y,r)= c_y+ r(\cos(\Omega_1) \cos(\theta)+i \sin(\theta)).$$ 
We set
$$\dot \gamma(\theta;c_y,r):= -r\cos(\Omega_1) \sin(\theta)+ i r \cos(\theta)$$ 
A phalanx with control sequence $(v_k)$ is tangent to an ellipse  on the complex plane
with center $ic_y$ and radii $(r \cos(\Omega_1),r)$ if and only if the following system of equations admits a solution $(t,\theta)$.
 \begin{equation}\label{stable}
\begin{cases}
\phi^{(i)}(t;(v_k),k) = \gamma(\theta; c_y,r) \quad & ~\text{ (Incidence Condition)}\\  
\arg(\phi^{(i)}(t;(v_k),k))=\arg(\dot \gamma(\theta;c_y,r)) \quad &~\text{ (Paralellism Condition)}\\
\omega\in[0,2\pi),~t\in[0,1]
\end{cases}
 \end{equation}
We introduce the scalar product on complex field $\cdot$, 
$$x\cdot y:=(\Re(x),\Im (x))\cdot (\Re(y),\Im (y)).$$
\begin{figure}
 \includegraphics[scale=0.4]{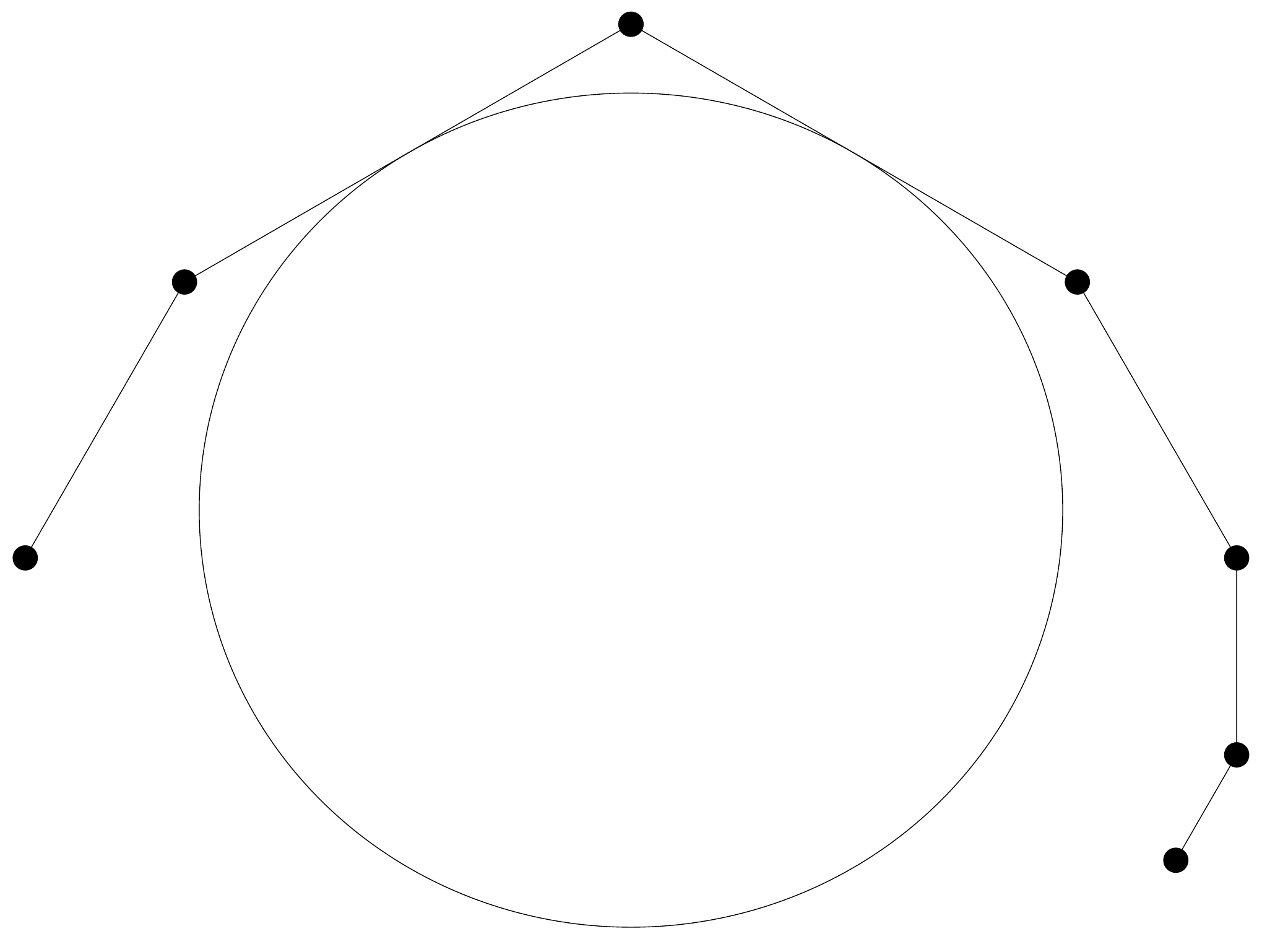}
\caption{A stable configuration using Finger 1 and Finger 3 with $\rho$ equal to the Golden Mean, $\omega=\pi/6$, $\Omega_1=\pi/12$. The control sequences are constantly equal to $1$, the bird's foot
lays on a cylinder of radius $r\sim 0.58$ with axis $x=0$, $z=-r$.}
\end{figure}

 By the reasonings above we have following proposition.
\begin{proposition}
 The bird's foot can be stable on a cylinder of center $c$ and radius $r$ using Finger 1 and Finger 3 if there exist two control sequences $(v^{(i)}_{k\leq K_i})$,  with $i=1,3$ and $K_1\leq 2$ and $K_3\leq4$, such that 
system (\ref{stable}) admits a solution.\\
If moreover $(v^{(i)}_{k\leq K_i})$, with $i=1,3$, also satisfy the grasping condition
\begin{equation}
\label{grasp}
 \left(\frac{t^{(1)}}{\rho^{K_1}}e^{-i\sum_{n=0}^{K_1} v^{(1)_n}w}\right) \cdot \left(\frac{t^{(3)}}{\rho^{K_3}}e^{-i\sum_{n=0}^{K_3} v^{(3)}_nw}\right) < 0  
\end{equation}
then the resulting configuration grasps the cylinder.
\end{proposition}
\begin{remark}
 Both terms of the scalar product in (\ref{grasp}) represent the orientation of the $K_i$-th phalanx of the $i$-th finger, the grasping condition follows thus directly by the definition of grasping configuration.
\end{remark}
\section{Owl's foot}
Owls have zygodactyly feet ensuring a good grasp when perching or clutching a pray. They are characterized by the further ability of rotating the a third toe to the front when in flight - see Figure \ref{owl}. 
In particular, when attacking the pray, the talons are spread out wide to increase the chance of a successful strike. 
Owl's feet are also endowed with the so called digital tendon locking mechanism (TLM), common among bats too. When an object, say a perch or a pray, touches the base of the foot then TLM engages and 
keeps the toes locked around the object without the need for the muscles to be contracted (Quinn and  Baumel (1990)).
\begin{figure}
 \subfloat[Zygodactily configuration]{\includegraphics[scale=0.7]{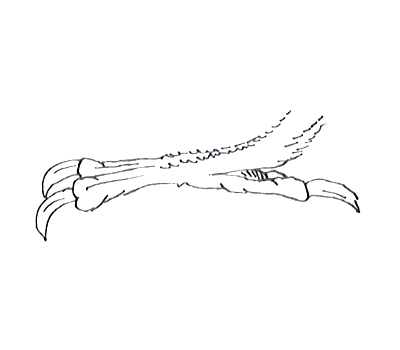}}\hskip0.5cm
 \subfloat[Isodactily configuration]{\includegraphics[scale=0.7]{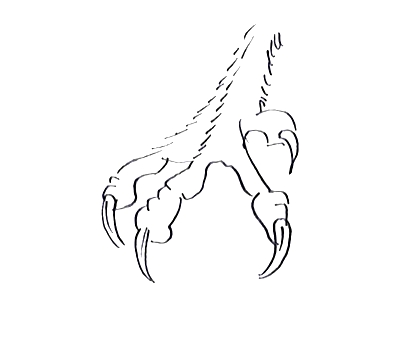}}
\caption{\label{owl} Owl's foot in two configurations, (A) is suitable for grasping, (B) is the typical flight configuration: as soon as a pray is targeted, phalanxes spread wide in order to 
maximize the chance of grabbing it.}
\end{figure}

From a mathematical point of view, TLM can be modeled as an hybrid system: when the first phalanx of any toe, namely the base of the foot, is not in contact with an object then the position of the phalanxes evolves according the control
dynamics described in previous section. If otherwise the first phalanx belongs to an appropriate region of $\RR^3$ denoted by $\mathcal O$ and representing an obstacle, say a pray or a branch, then TLM engages.
Since TLM is not a voluntary movement, then the corresponding dynamic is not controlled - see Figure.
\ref{owl}-(A).

\begin{figure}
\subfloat[TLM engagement for Finger 4]{ \includegraphics[scale=0.6]{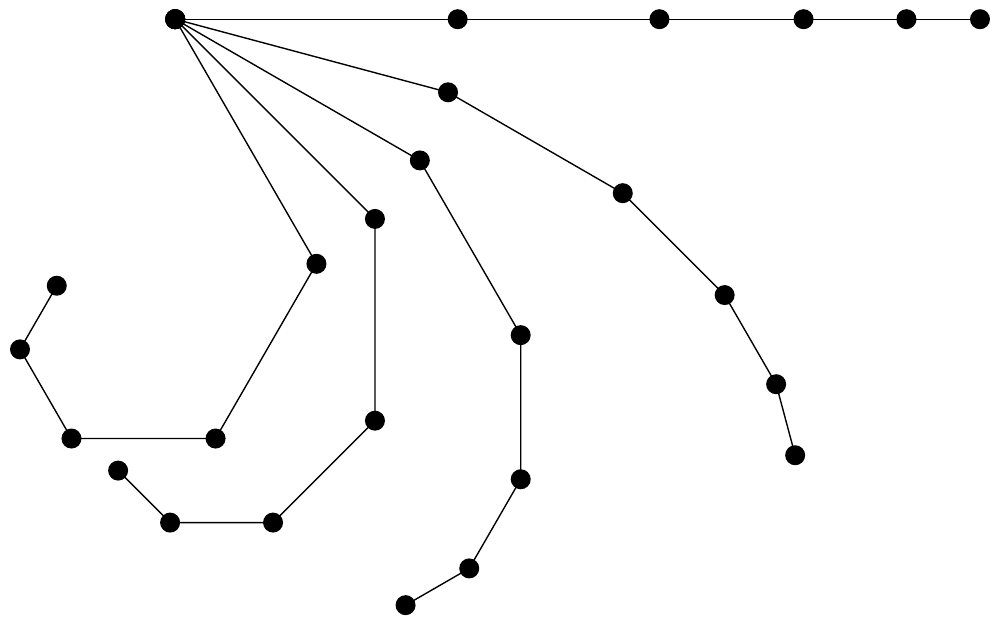}}\hskip0.5cm
\subfloat[Talon (i.e. extremal of the toe) trajectory]{ \includegraphics[scale=0.6]{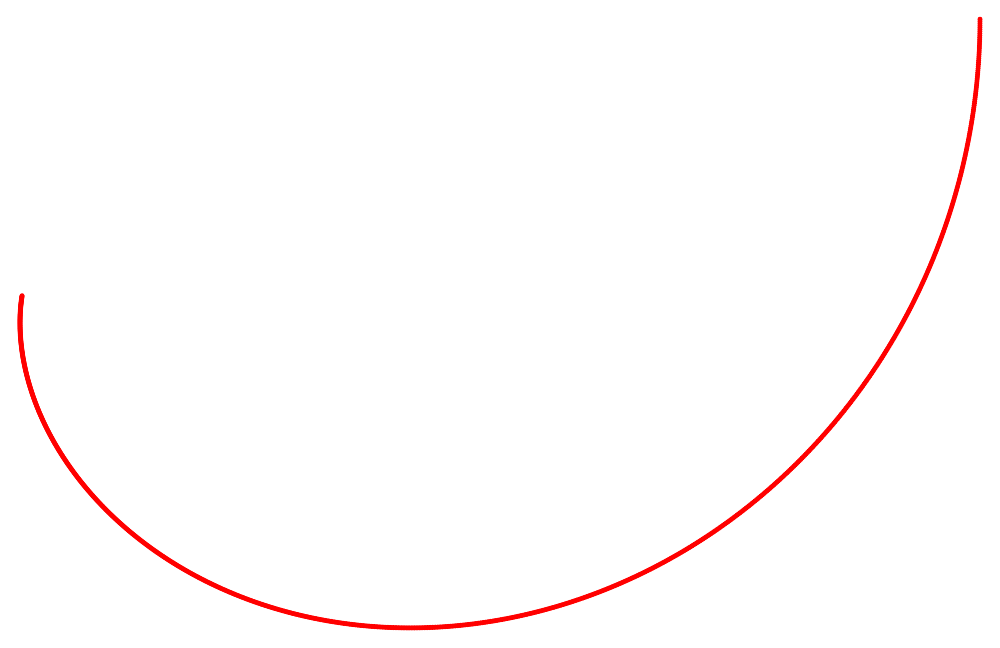}}
\caption{\label{tlm} Various stages of TLM engagement: all rotational controls are equal to $v(t)\in[0,\pi/6]$, $t\in[0,T]$, where $T$ is engagement time. }
\end{figure}
As in previous sections, let $x_{k}^{(h)}$ be the position of the $k$-th junction of the $h$-th toe of owl's foot and $u_{k}^{(h)}\in[0,1]$ be the corresponding rotation control. In our model phalanxes
are segments, in particular the first phalanx of the $h$-th toe is the set
$$P^{(h)}:=\left\{ a x_{1}^{(h)}\mid a \in [0,1]\right\}.$$
Also define the TLM map $v(t):[0,T]\mapsto [0,1]$ where $T$ is the is the engagement time, that is the time requested to the toes to contract when it touches an object $\mathcal O$, and  $v(t)$ is a continuous
 non-decreasing map. We have for every $h=1,\dots,4$
\begin{equation}
\begin{cases}
\displaystyle{x_{k}^{(h)}(t)= \sum_{j=1}^{k+1} \frac{1}{\rho^j} e^{-i\omega \sum_{n=1}^j u_{n}^{(h)}(t)}} \quad & \text{if } P^{(l)}\cap \mathcal O =\emptyset; l=1,\dots,4\\
\displaystyle{x_{k}^{(h)}(t)= \sum_{j=1}^{k+1} \frac{1}{\rho^j} e^{-i\omega j v(t)}}\quad & \text{otherwise}
\end{cases}
\end{equation}

\begin{remark}
 If $\mathcal O$ is convex then  $P^{(l)}\cap \mathcal O \not=\emptyset$ for some $l=1,\dots,4$ is satisfied for every $t\in[0,T]$. 
At time $T$ the tendon is locked and no further movement is allowed.
This is indeed the goal of such mechanism: to keep the toes contracted without the contribution of muscles. We finally remark that disengagement of TLM is yet  an open problem among biologists and zoologists.

\end{remark}

\end{document}